\documentclass[12pt]{article}
\usepackage{graphicx}
\usepackage{amsmath,amsfonts,latexsym,amscd,amssymb,theorem}
\usepackage{graphicx}
\usepackage[ansinew]{inputenc}
\usepackage[shortlabels]{enumitem}

\usepackage{epsfig}
\usepackage{amsmath}
\usepackage{amssymb}
\usepackage{afterpage}
\usepackage{selinput}

\newcommand{\ba}{\begin{eqnarray}}
\newcommand{\ea}{\end{eqnarray}}

\newtheorem{thm}{Theorem}[section]

\newtheorem{conjecture}{Conjecture}

\newtheorem{theorem}[thm]{Theorem}

\newtheorem{lemma}[thm]{Lemma}

\newtheorem{claim}[thm]{Claim}

\makeatletter
\newcommand*{\rom}[1]{\expandafter\@slowromancap\romannumeral #1@}
\makeatother

\begin{document}
\title{\textbf{Forests and the Strong Erd\"{o}s-Hajnal Property}}
\maketitle


\begin{center}
\author{Soukaina ZAYAT \footnote{Department of Mathematics, Lebanese University, Hadath, Lebanon. (soukaina.zayat.96@outlook.com)}}
\end{center}

\begin{abstract}
 An equivalent directed version of the celebrated unresolved conjecture of Erd\"{o}s and Hajnal proposed by Alon et al. states that for every tournament $H$ there exists $ \epsilon(H) > 0 $ such that every $H$-free $n$-vertex tournament $T$ contains a transitive subtournament of order at least $ n^{\epsilon(H)} $. A tournament H has the strong EH-property if there exists $c > 0$ such that for every $H$-free tournament $T$ with $\mid$$T$$\mid$ $ > 1$, there
exist disjoint vertex subsets $A$ and $B$, each of cardinality at least $c$$\mid$$T$$\mid$ and every vertex of $A$
is adjacent to every vertex of $B$. Berger et al. proved that  the unique five-vertex tournament denoted by $C_5$, where every vertex has two inneighbors and two outneighbors has the strong EH-property. It is known that every tournament with the strong EH-property also has the EH-property. In this paper we construct an infinite class of tournaments $-$ the so-called spiral galaxies and we prove that every spiral galaxy has the strong EH-property.   
\end{abstract}

\section{Introduction}
Let $ G $ be an undirected graph. A \textit{clique} in $G$ is a set of pairwise adjacent vertices and a \textit{stable set} in $G$ is a set of pairwise nonadjacent vertices.  A \textit{tournament} is an orientation of a complete graph. A tournament is \textit{transitive} if it contains no directed cycle. Let $H$ be a tournament.  If $(u,v)\in A(H)$, then we say that $u$ is \textit{adjacent to} $v$, and we write $u\rightarrow v$. In this case, we also say that $v$ is \textit{adjacent from} $u$, and we write $v\leftarrow u$. For two disjoint sets of vertices $V_{1},V_{2}$ of $H$, we say that $V_{1}$ is \textit{complete to} $V_{2}$ (equivalently $V_{2}$ is \textit{complete from} $V_{1}$) if every vertex of $V_{1}$ is adjacent to every vertex of $V_{2}$. We say that a vertex $v$ is \textit{complete to} (resp. \textit{from}) a set $V$ if $\lbrace v \rbrace$ is complete to (resp. from) $V$, and we write $v \rightarrow V$ (resp. $v \leftarrow V$).   Let $X \subseteq V(H)$. The \textit{subtournament of} $H$ \textit{induced by} $X$ is denoted by $H$$\mid$$X$. Let $S$ be a tournament.  We say that $H$ \textit{contains} $S$ if $S$ is isomorphic to $H$$\mid$$X$ for some $X \subseteq V(H)$. If $H$ does not contain $S$, we say that $H$ is $S$-$free$.\vspace{3mm}\\   
In 1989 Erd\"{o}s and Hajnal proposed the following conjecture \cite{jhp} (EHC):
\begin{conjecture} For any undirected graph $H$ there exists $ \epsilon(H) > 0 $ such that every $n$-vertex undirected graph that does not contain $H$ as an induced subgraph contains a clique or a stable set of size at least $ n^{\epsilon(H)}. $
\end{conjecture}
In 2001 Alon et al. proved \cite{fdo} that Conjecture $1$ has an equivalent directed version, as follows:
\begin{conjecture} \label{a} For any tournament $H$ there exists $ \epsilon(H) > 0 $ such that every $ H $-free tournament with $n$ vertices contains a transitive subtournament of order at least $ n^{\epsilon(H)}. $
\end{conjecture}
A tournament $H$ \textit{has the Erd\"{o}s-Hajnal property (EH-property)} if there exists $ \epsilon(H) > 0 $ such that every $ H $-free tournament $T$ with $n$ vertices contains a transitive subtournament of size at least $ n^{\epsilon(H)}. $ \\

Let $ \theta = (v_{1},...,v_{n}) $ be an ordering of the vertex set $V(T)$ of an $ n $-vertex tournament $T$. We say that a vertex $ v_{j} $ is \textit{between two vertices $ v_{i},v_{k} $ under} $ \theta = (v_{1},...,v_{n}) $ if $ i < j < k $ or $ k < j < i $. An arc $ (v_{i},v_{j}) $ is a \textit{backward arc under} $ \theta $ if $ i > j $. The \textit{set of backward arcs of $T$ under} $ \theta $ is denoted by $A_T(\theta) $. The \textit{backward arc digraph of $T$ under} $ \theta $, denoted by $ B(T,\theta) $, is the subdigraph that has vertex set $V(T)$ and arc set $A_T(\theta)$.  We say that \textit{$V(T)$ is the disjoint union of $X_1$,...,$X_t$ under $\theta$} if $V(T)$ is the disjoint union of $X_1$,...,$X_t$ and $A(B(T,\theta ))= \displaystyle{\bigcup_{i=1}^{t}}A(B(T|X_i,\theta_{i}))$, where $\theta_{i}$ is the restriction of $\theta$ to $X_i$.  

A tournament $S$ on $p$ vertices with $V(S)= \lbrace u_{1},u_{2},...,u_{p}\rbrace$ is a \textit{right star} (resp. \textit{left star}) (resp. \textit{middle star}) if there exists an ordering $\beta = (u_{1},u_{2},...,u_{p})$ of its vertices such that $A(B(S,\beta ))=\{ (u_{p},u_{i}): i=1,...,p-1\}$ (resp. $A(B(S,\beta ))=\{ (u_{i},u_1): i=2,...,p\}$) (resp. $A(B(S,\beta ))=\{ (u_{i},u_m): i=m+1,...,p\}\cup \{ (u_{m},u_{i}): i=1,...,m-1\}$, where $m\in\{2,...,p-1\}$). In this case we write $S = \lbrace u_{1},u_{2},...,u_{p}\rbrace$ and we call $\beta$ a \textit{right star ordering} (resp. \textit{left star ordering}) (resp. \textit{middle star ordering}) of $S$, $u_{p}$ (resp. $u_{1}$) (resp. $u_{m}$) the \textit{center of} $S$, and $u_{1},...,u_{p-1}$ (resp. $u_{2},...,u_{p}$)(resp. $u_{1},...,u_{m-1},u_{m+1},...,u_p$)   the \textit{leaves of} $S$. A \textit{frontier star} is a left star or a right star. A \textit{star} is a middle star or a frontier star. A \textit{star ordering} is a right or left or middle star ordering.
 
 A \textit{star} $S=\lbrace v_{i_{1}},...,v_{i_{t}}\rbrace$ \textit{of $T$ under $\theta$} (where $i_{1}<...<i_{t}$) is the subtournament of $T$ induced by $\lbrace v_{i_{1}},...,v_{i_{t}}\rbrace$ such that $S$ is a star and $S$ has the star ordering $ (v_{i_{1}},...,v_{i_{t}})$ under $\theta$ (i.e $(v_{i_{1}},...,v_{i_{t}})$ is the restriction of $\theta$ to $V(S)$ and $ (v_{i_{1}},...,v_{i_{t}})$ is a star ordering of $S$).  \vspace{2mm} 

Let $T$ be a tournament and assume that there exists an ordering $ \theta $ of its vertices such that $V(T)$ is the disjoint union of $V(S_{1}),...,V(S_{l}),X$ under $\theta$, where $S_{1},...,S_{l}$ are the stars of $T$ under $\theta$, and for every $x\in X$, $\lbrace x \rbrace$ is a singleton component of $B(T,\theta)$. In this case $T$ is called \textit{nebula} and $\theta$ is called a \textit{nebula ordering}. If all the stars of $T$ under $\theta$ are frontier stars and no center of a star is between leaves of another star under $\theta$, then $\theta$ is called a \textit{galaxy ordering} and $T$ is called a \textit{galaxy} under $\theta$. If moreover $X$ is empty, then $T$ is called a \textit{regular galaxy} under $\theta$. 

In \cite{polll} Berger et al. proved that every galaxy has the EH-property, and in \cite{kg,SZ,SZG} Conjecture \ref{a} was proved for more general classes of tournaments.

Let $T$ be a tournament. A \textit{pure pair} in $T$ is an ordered pair $(A,B)$ of disjoint subsets of $V(T)$ such
that every vertex in $A$ is adjacent to every vertex in $B$, and its order denoted by $\mathcal{O}(A,B)$ is $min(\mid$$A$$\mid , \mid$$B$$\mid)$. Define $\mathcal{P}(T):= max\{ \mathcal{O}(A,B): (A,B)$ is a pure pair of $T\}$. We call $T$ $ \alpha$-\textit{coherent} for $ \alpha > 0 $ if $\mathcal{P}(T) < $ $\alpha \mid $$T$$\mid$.\vspace{2mm}\\
A tournament $H$ has the \textit{strong EH-property} if there exists $\alpha > 0$ such that for every $H$-free
tournament $T$ with $\mid$$T$$\mid > 1$, we have: $\mathcal{P}(T)\geq \alpha\mid$$T$$\mid$. It is easy to see that for
every tournament with the strong EH-property also has the EH-property \cite{polll,SEH}.\vspace{3mm}\\
In \cite{SEH} Berger et al. proved that  the unique five-vertex tournament denoted by $C_5$, where every vertex has two inneighbors and two outneighbors has the strong EH-property. In \cite{chudnovsky} Chudnovsky et al. asked if it might be true that a tournament $H$ has the strong $EH$-property if and only if its vertex set has an ordering in which the backward arc digraph of $H$ under this ordering is a forest. Chudnovsky et al. proved that the necessary condition is true:
\begin{theorem}\cite{chudnovsky}
If a tournament has the strong EH-property then there is an ordering of $V(H)$ for which the
backward arc digraph is a forest.
\end{theorem}
Unfortunately, the following is still open:
\begin{conjecture}\label{forest}
If a tournament $H$ has an ordering of its vertices for which the backward arc digraph of $H$ under this ordering is a forest, then $H$ has the strong EH-property.
\end{conjecture}
In \cite{chudnovsky} Chudnovsky et al. made a small step towards Conjecture \ref{forest} by showing that if a tournament $H$ is $C_5$-free and it has an  ordering $\theta$ of its vertices such that $\mid$$A_H(\theta)$$\mid \leq 3$, then it has the strong EH-property. In particular, they proved that every tournament with at most six vertices has the property, except for three six-vertex tournaments that they could not decide.\vspace{3mm}\\
In \cite{SZG} the author and Ghazal proved that every galaxy with spiders has the EH-property (see \cite{SZG} for the detailed description of galaxies with spiders).\vspace{4mm}

 In this paper we prove Conjecture \ref{forest} for an infinite family of tournaments $-$ the so-called \textit{spiral galaxies}. Spiral galaxies are nebulas satisfying some conditions to be explained in details in Section \ref{result}. Also note that every spiral galaxy is a galaxy with spiders.

\section{Smooth $(c,\lambda ,w)$-structure}
Let $T$ be a tournament and let $X, Y \subseteq V(T)$ be disjoint. Denote by $e_{X,Y}$ the number of directed arcs $(x,y)$, where $x \in X$ and $y \in Y$. The \textit{directed density from $X$ to} $Y$ is defined as $d(X,Y) = \frac{e_{X,Y}}{\mid X \mid.\mid Y \mid} $. 
\begin{lemma} \cite{polll} \label{s} Let $A_{1},A_{2}$ be two disjoint sets such that $d(A_{1},A_{2}) \geq 1-\lambda$ and let $0 < \eta_{1},\eta_{2} \leq 1$. Let $\widehat{\lambda} = \frac{\lambda}{\eta_{1}\eta_{2}}$. Let $X \subseteq A_{1}, Y \subseteq A_{2}$ be such that $\mid$$X$$\mid$ $\geq \eta_{1} \mid$$A_{1}$$\mid$ and $\mid$$Y$$\mid$ $\geq \eta_{2} \mid$$A_{2}$$\mid$. Then $d(X,Y) \geq 1-\widehat{\lambda}$. 
\end{lemma}
The following is introduced in \cite{bnmm}.\\
Let $ c > 0, 0 < \lambda < 1 $ be constants, and let $w$ be a $ \lbrace 0,1 \rbrace  $-vector of length $ \mid $$w$$ \mid $. Let $T$ be a tournament with $ \mid $$T$$ \mid$ $ = n. $ A sequence of disjoint subsets $ \chi = (A_{1}, A_{2},..., A_{\mid w \mid}) $ of $V(T)$ is a smooth $ (c,\lambda, w) $-structure if:\\
$\bullet$ whenever $ w_{i} = 0 $ we have $ \mid $$A_{i}$$ \mid$ $ \geq cn $ (we say that $ A_{i} $ is a \textit{linear set}).\\
$\bullet$ whenever $ w_{i} = 1 $ the tournament $T$$\mid$$ A_{i} $ is transitive and $ \mid $$A_{i}$$ \mid$ $ \geq c.tr(T) $ (we say that $ A_{i} $ is a \textit{transitive set}).\\
$\bullet$ $ d(\lbrace v \rbrace, A_{j}) \geq 1 - \lambda $ for $v \in A_{i} $ and $ d(A_{i}, \lbrace v \rbrace) \geq 1 - \lambda $ for $v \in A_{j}, i < j $ (we say that $\chi$ is \textit{smooth}).
\begin{theorem} \cite{SEH} \label{t}
Let $S$ be a $f$-vertex tournament, let $w$ be an all-zero vector. Then there exist $ c > 0 $ such that every $ S $-free tournament contains a smooth $ (c,\frac{1}{f},w) $-structure.
\end{theorem}
Let $(S_{1},...,S_{\mid w \mid})$ be a smooth $(c,\lambda ,w)$-structure of a tournament $T$, let $i \in \lbrace 1,...,\mid$$w$$\mid \rbrace$, and let $v \in S_{i}$. For $j\in \lbrace 1,2,...,\mid$$w$$\mid \rbrace \backslash \lbrace i \rbrace$, denote by $S_{j,v}$ the set of the vertices of $S_{j}$ adjacent from $v$ for $j > i$ and adjacent to $v$ for $j<i$.
\begin{lemma} \label{g}\cite{SZ} Let $0<\lambda<1$, $0<\gamma \leq 1$ be constants and let $w$ be a $\lbrace 0,1 \rbrace$$-$vector. Let $(S_{1},...,S_{\mid w \mid})$ be a smooth $(c,\lambda ,w)$-structure of a tournament $T$ for some $c>0$. Let $j\in \lbrace 1,...,\mid$$w$$\mid \rbrace$. Let $S_{j}^{*}\subseteq S_{j}$ such that $\mid$$S_{j}^{*}$$\mid$ $\geq \gamma \mid$$S_{j}$$\mid$, and let $A= \lbrace x_{1},...,x_{k} \rbrace \subseteq \displaystyle{\bigcup_{i\neq j}S_{i}}$ for some positive integer $k$. Then $\mid$$\displaystyle{\bigcap_{x\in A}S^{*}_{j,x}}$$\mid$ $\geq (1-k\frac{\lambda}{\gamma})\mid$$S_{j}^{*}$$\mid$. In particular $\mid$$\displaystyle{\bigcap_{x\in A}S_{j,x}}$$\mid$ $\geq (1-k\lambda)\mid$$S_{j}$$\mid$.
\end{lemma}
\begin{proof}
The proof is by induction on $k$. without loss of generality assume that $x_{1} \in S_{i}$ and $j<i$. Since $\mid$$S_{j}^{*}$$\mid$ $\geq \gamma \mid$$S_{j}$$\mid$, then by Lemma \ref{s}, $d(S^{*}_{j},\lbrace x_{1}\rbrace) \geq 1-\frac{\lambda}{\gamma}$. So $1-\frac{\lambda}{\gamma} \leq d(S^{*}_{j},\lbrace x_{1}\rbrace) = \frac{\mid S^{*}_{j,x_{1}}\mid}{\mid S_{j}^{*}\mid}$. Then $\mid$$S^{*}_{j,x_{1}}$$\mid$ $\geq (1-\frac{\lambda}{\gamma})$$\mid$$S_{j}^{*}$$\mid$ and so true for $k=1$.
Suppose the statement is true for $k-1$.\\ $\mid$$\displaystyle{\bigcap_{x\in A}S^{*}_{j,x}}$$\mid$ $=\mid$$(\displaystyle{\bigcap_{x\in A\backslash \lbrace x_{1}\rbrace}S^{*}_{j,x}})\cap S^{*}_{j,x_{1}}$$\mid$ $= \mid$$\displaystyle{\bigcap_{x\in A\backslash \lbrace x_{1}\rbrace}S^{*}_{j,x}}$$\mid$ $+$ $\mid$$S^{*}_{j,x_{1}}$$\mid$ $- \mid$$(\displaystyle{\bigcap_{x\in A\backslash \lbrace x_{1}\rbrace}S^{*}_{j,x}})\cup S^{*}_{j,x_{1}}$$\mid$ $\geq (1-(k-1)\frac{\lambda}{\gamma})\mid$$S_{j}^{*}$$\mid$ $+$ $(1-\frac{\lambda}{\gamma})\mid$$S_{j}^{*}$$\mid$ $-$ $\mid$$S_{j}^{*}$$\mid$ $= (1-k\frac{\lambda}{\gamma})\mid$$S_{j}^{*}$$\mid$. $\hfill{\square}$       
\end{proof}
\section{Main Result}\label{result}
Let $H$ be a tournament, let $\theta =(1,...,t)$ be an ordering of its vertices, and let $2\leq r \leq t-5$. $H$ is called \textit{middle-pair-star} under $\theta =(1,...,t)$ if either $A_T(\theta )=\{(r+3,r),(r+4,r+1)\}\cup \{(r,i): i=1,...,r\}\cup \{(i,r+4): i=r+5,...,t\}$ or $A_T(\theta )=\{(r+3,r),(r+4,r+1)\}\cup \{(r+4,i): i=1,...,r\}\cup \{(i,r): i=r+5,...,t\}$. In this case we call the vertices $r,...,r+4$ the \textit{golden vertices of $H$}, $r,r+4$ the \textit{centers of $H$}, and $1,...,r-1,r+5,...,t$ the \textit{leaves of $H$}. $H$ is a \textit{left-pair-star} under $\theta$ if $A_H(\theta )=\{(4,1),(5,2)\}\cup \{(i,5):i=6,...,j\}\cup \{(i,1):i=j+1,...,t\}$ or $A_H(\theta )=\{(4,1),(5,2)\}\cup \{(i,1):i=6,...,j\}\cup \{(i,5):i=j+1,...,t\}$, where $6\leq j\leq t-1$. In this case we call the vertices $1,...,5$ the \textit{golden vertices of $H$}, $1,5$ the \textit{centers of $H$}, and $6,...,t$ the \textit{leaves of $H$}. \textit{Right-pair-stars} are defined similarly. If $H$ is a middle-pair-star (resp. right-pair-star) (resp. left-pair-star) under $\theta$, then $\theta$ is called \textit{star ordering of} $H$.
\\Let $T$ be a tournament drawn under an ordering $\theta = (v_1,...,v_n)$ of its vertices.  A \textit{pair-star $\mathcal{P}:=\{v_{i_1},...,v_{i_t}\}$ of $T$ under $\theta$} (where $i_1<...<i_t$) is an induced subtournament of $T$ with vertex set $\{v_{i_1},...,v_{i_t}\}$, such that: $(v_{i_1},...,v_{i_t})$ is its star ordering, the golden vertices are consecutive under $\theta$, and leaves incident to the same center are consecutive under $\theta$.\vspace{3mm}\\
A tournament $H$ is a \textit{path-galaxy} under an ordering $\theta$ of its vertices if $H$ is a galaxy under $\theta$, and the leaves of every star are consecutive under $\theta$. A tournament $H$ is a \textit{regular path-galaxy under $\theta$} if $H$ is a regular galaxy under $\theta$ which is a path-galaxy, and moreover all the stars of $H$ under $\theta$ are of the same size. 
Let H be a tournament and let $\theta$ be an ordering of its vertices such that $V (H)$ is the disjoint
union of $V (S_1)$, ...,$V (S_l), X$ under $\theta$, where $S_1, ..., S_l$ are the pair-stars of $H$ under $\theta$, $H$$\mid$$X$ is a path-galaxy under $\theta_{X}$ where $\theta_{X}$ is the restriction of $\theta$ to $X$, leaves of every star of $H|X$ under $\theta_{X}$ are also consecutive under $\theta$,  and for any given
$1 \leq i < j \leq l$ every vertex of $V(S_i) $ is before every vertex of $V (S_j)$ under $\theta$. We call this ordering a spiral galaxy ordering of $H$ and we call $H$ a \textit{spiral galaxy under $\theta$}. Let $Q_1,...,Q_t$ be the stars of $H$$\mid$$X$ under $\theta_{X}$. If moreover all the following are satisfied:\vspace{2.5mm}\\
$\bullet$ Centers of $S_i$ are incident to  same number of leaves for $i=1,...,l$ and $|S_i|=|S_j|$ for all $i\neq j$.\vspace{1.5mm}\\
$\bullet$ $t=l$ and $2|Q_i|= |S_i|-3$ for $i=1,...,l$,\vspace{2mm}\\  then we call $H$ a \textit{uniform spiral galaxy under} $\theta$. \vspace{4mm}\\ 
In this section  we start by proving a structural property of $\alpha$-coherent tournaments. Then we prove that every spiral galaxy has the strong EH-property.

\begin{lemma} \label{q}
Let $c,\alpha > 0$ be constants, where $\alpha\leq \frac{c}{m+1}$. Let $T$ be an $\alpha $-coherent tournament with $\mid$$T$$\mid$ $ =n$, and   let $A,B_{1},...,B_{m}$ be $m$ vertex disjoint subsets of $V(T)$ with $ \mid $$A$$ \mid$ $ \geq cn$ and $ \mid $$B_{i}$$ \mid$ $\geq cn$ for $i=1,...,m$. 
Then there exist vertices $b_{1},...,b_{m},a$ such that $a\in A$, $b_{i}\in B_{i}$ for $i=1,...,m$, and $\lbrace a \rbrace$ is complete to $\lbrace b_{1},...,b_{m} \rbrace$. Similarly there exist vertices $u_{1},...,u_{m},r$ such that $r\in A$, $u_{i}\in B_{i}$ for $i=1,...,m$, and $\lbrace r \rbrace$ is complete from $\lbrace u_{1},...,u_{m} \rbrace$. Similarly there exist vertices $s_{1},...,s_{t},...,s_m,r$ such that $r\in A$, $s_{i}\in B_{i}$ for $i=1,...,m$, and $\{s_1,...,s_t\}\leftarrow\lbrace r \rbrace\leftarrow \{s_{t+1},...,s_m\}$.
\end{lemma}
\begin{proof}
We will prove only the first statement  because the latter can be proved analogously.
Let $A_{i} \subseteq A$ such that $A_{i}$ is complete from $B_{i}$ for $i = 1,...,m$.
Since $T$ is $ \alpha $-coherent, then $\mid$$A_i$$\mid < \alpha n$, and so $A^*:=$ $\mid$$A\backslash (\bigcup_{i=1}^{m}A_{i})$$\mid \geq cn-m\alpha n\geq (m+1)\alpha n- m\alpha n= \alpha n$. This implies that $A^*$ is non-empty. Fix $a\in A$.  So there exist vertices $b_{1},...,b_{m}$ such that $\lbrace a \rbrace$ is complete to $\lbrace b_{1},...,b_{m} \rbrace$. $\hfill {\square}$ 
\end{proof}
\vspace{5mm}\\ The following well-known theorem will be very useful in our latter analysis:
\begin{theorem}\cite{stearns}\label{transitive}
Every tournament on $2^{k-1}$ vertices contains a transitive subtournament of size at least $k$.
\end{theorem}
An ordering $\alpha$ of the vertex set of a transitive tournament $S$ is called \textit{transitive ordering of $S$} if $A_S(\alpha)=\phi$. Let $\alpha^{'}$ be an ordering of $V(S)$. Define the \textit{transitive operation} to be the permutation of vertices that transforms $\alpha^{'}$ to $\alpha$  (note possibly $\alpha =\alpha^{'}$).\vspace{4mm}\\
 Let $H$ be a regular path-galaxy tournament under an ordering $\theta = (u_1,...,u_h)$ of its vertices. Let $Q_1,...,Q_l$ be the stars of $H$ under $\theta$. 
 Let $w$ be an all-zero vector. We say that a smooth $(c,\lambda ,w)$-structure of a tournament $T$ \textit{corresponds}  \textit{to $H$} if  $\mid$$w$$\mid =h$. Define $\Pi_{H}(\theta)$ to be the set of orderings of $H$ that can be transformed to $ \theta$  by applying the transitive operation to the leaves of all stars of $H$.

Let $\chi :=(A_{1},...,A_{\mid w \mid})$ be a smooth $(c,\lambda ,w)$-structure in a tournament $T$ that corresponds to $H$. We say that $H$ is \textit{well-contained in} $\chi$ if there exists $x_i\in A_i$ for $i=1,...,\mid$$w$$\mid$ such that $T$$\mid$$\{x_1,...,x_{\mid w\mid }\}$ is isomorphic to $H$ and $(x_1,...,x_{\mid w \mid })$ is one of the orderings of $H$ in $\Pi_{H}(\theta)$.

\begin{theorem}\label{mm}
Every spiral galaxy has the strong EH-property.
\end{theorem}
\begin{proof}
Let $H$ be a spiral galaxy tournament under an ordering $\theta = (u_1,...,u_h)$ of its vertices. We can assume that $H$ is a uniform spiral galaxy, since every spiral galaxy is a subtournament of a uniform spiral galaxy.   Let $S_1,...,S_l$ be the pair-stars  of $H$ under $\theta$, such that $|S_i|=2b+5$. Let $T$ be an $H$-free tournament on $n$ vertices.  Let $\xi = 2^b+5$ and let $\nu =3l(2^{b-1}+1)$. Let $w:= (0,...,0)$ with $\mid$$w$$\mid$ $= \nu(l\xi +1)+l\xi$ be an all-zero vector. Theorem \ref{t} implies that there exists $c>0$ and a smooth $(c,\lambda ,w)$-structure $(A_1,...,A_{\mid w\mid })$ denoted by $\chi$, with $\lambda =\frac{1}{h}$.  We are going to prove that there exists a pure pair $(A,B)$ in $T$ with $\mathcal{O}(A,B)\geq \alpha cn$, where $\alpha = min\{\frac{1}{30(2^{b-1}+2)},\frac{b+1}{h(2^{b-1}+1)}\}$. Assume that the contrary is true. Then $\mathcal{P}(T)<\alpha n$. So $T$ is $\alpha$-coherent. For $k \in \lbrace 1,...,l \rbrace$, define $H^{k} = H$$\mid$$\bigcup_{j=1}^{k} V(S_{j})$.   Select the sets $A_{i_1},...,A_{i_{l\xi}}$ in $\chi$, such that $i_1=\nu +1$, and $i_{j+1}-i_{j}=\nu +1$ for all $j\in \{1,...,l\xi -1\}$. Rename $A_{i_1},...,A_{i_{l\xi}}$ by $W_1,...,W_{l\xi}$. Clearly $\tilde{\chi}=(W_1,...,W_{l\xi})$ is a smooth $(c,\lambda ,\tilde{w})$-structure, where $\mid$$\tilde{w}$$\mid$ $=l\xi$. Let $\theta_{k}:=(u_{k_1},...,u_{k_{q_k}})$ be the restriction of $\theta$ to $V(H^k)$, where $k\in\{1,...,l\}$ and $q_k=k(2b+5)=\mid$$H^k$$\mid$. For $k=l$, let $q=q_l$. Rename now the vertices in the ordering $\theta_{k}$ by $v_1,...,v_{q_k}$ (that is $\theta_{k}:=(u_{k_1},...,u_{k_{q_k}})=(v_1,...,v_{q_k})$). Clearly $H^k$ is a nebula under $\theta_{k}$. Let $R_1,....,R_{2l}$ be the stars of $H^{l}$ under $\theta_{l}$, and let $L_1,...,L_{2l}$ be the set of leaves of $R_1,....,R_{2l}$ respectively. For all $i\in \{1,...,2l\}$, let $a_i\in L_i$ be a golden vertex of some pair star. Define $L_i':= L_i\backslash \{a_i\}$ and $E_i$ to be the set of arcs between $L_i'$ and $\{a_i\}$ for $i=1,...,2l$.  Let $\mathcal{H}^k$ be the digraph obtained from $H^k$ by deleting all the arcs  in  $\displaystyle{\bigcup_{i=1}^{2k}}A(H^k|L_i)$. When $k=l$, write $\mathcal{H}$ instead of $\mathcal{H}^l$. Let us call $\theta_{k}$ the \textit{forest ordering} of $\mathcal{H}^k$ and for $i=1,...,2l$, call $L_i'$ a set of \textit{sister leaves} of $\mathcal{H}^k$ under $\theta_{k}$ (we call them sister leaves because they are consecutive leaves under $\theta$). 
\begin{claim}\label{cl1}
For some $j_1,...,j_q$ with $1\leq j_1<...<j_q \leq l\xi$, there exist vertices $a_i\in W_{j_i}$ for $i=1,...,q$, such that \begin{itemize}
\item $T':=T$$\mid$$\{a_1,...,a_{q}\}$ contains a copy of $\mathcal{H}$, denoted by $\mathcal{H}'$, and
\item $(a_1,...,a_{q})$ is the forest ordering of $\mathcal{H}$, and
\item every subtournament in $T'$ induced by a set of sister leaves of $\mathcal{H}'$ under  $(a_1,...,a_{q})$, is a transitive subtournament. 
\end{itemize} 
\end{claim}
\noindent \sl {Proof of Claim \ref{cl1}. }\upshape 
Assume that $S_1$ is a middle-pair-star and $\{v_1,...,v_b\}\leftarrow v_{b+1}$ (else the argument is similar and we omit it). Let $\rho =2^{b-1}$. As $\alpha < \frac{c}{\rho +2}$ and $W_{i}\geq cn$ for $i=1,...,\mid$$\tilde{w}$$\mid$, Lemma \ref{q} implies that there exist $y_i\in W_i$ for $i=1,...,\rho+1,\rho +4$, such that $\{y_1,...,y_{\rho }\}\leftarrow y_{\rho +1}\leftarrow y_{\rho +4}$. By Theorem \ref{transitive}, $T|\{y_1,...,y_{\rho }\}$ contains a transitive subtournament of size $b$. Let $T_1:=T|\{y_{j_1},...,y_{j_b}\}$ be this transitive subtournament, with $1 \leq j_1<...<j_b\leq \rho$. Let $W_{\rho +2}^*=\{y\in W_{\rho +2}:V(T_1)\cup\{y_{\rho +1}\}\rightarrow y \rightarrow y_{\rho +4}\}$ and let $W_i^*=\{y\in W_i:V(T_1)\cup\{y_{\rho +1},y_{\rho +4}\}\rightarrow y\}$ for $i=\rho +5,...,\xi$. Lemma \ref{g} implies that $\mid$$W_i^*$$\mid$ $\geq (1-\frac{b+2}{h})\mid$$W_i$$\mid$ $\geq \frac{\mid W_i\mid}{2}\geq\frac{c}{2}n$ for $i=\rho +2,\rho +5,...,\xi$. Now since $\alpha < \frac{c}{2(\rho +2)}$ and $W^*_{i}\geq \frac{c}{2}n$ for $i=\rho +2,\rho +5,...,\xi$, then by Lemma \ref{q}, there exist $y_i\in W^*_i$ for $i=\rho +2,\rho +5,...,\xi$, such that $y_{\rho +2}\leftarrow y_{\rho +5}\leftarrow\{y_{\rho +6},...,y_{\xi}\}$. By Theorem \ref{transitive}, $T|\{y_{\rho +6},...,y_{\xi}\}$ contains a transitive subtournament of size $b$. Let $T_2:=T|\{y_{j_{b+6}},...,y_{j_{2b+5}}\}$ be this transitive subtournament. Let $W_{\rho +3}^*=\{y\in W_{\rho +3}: V(T_1)\cup\{y_{\rho +1},y_{\rho +2}\}\rightarrow y\rightarrow \{y_{\rho +4},y_{\rho +5}\}\cup V(T_2)\}$. Similarly, we prove that $W^*_{\rho +3}\geq \frac{c}{3}n$, and so $W^*_{\rho +3}\neq\phi$. Fix $y_{\rho +3}\in W_{\rho +3}^*$. So $T$$\mid$$\{y_{j_1},...,y_{j_b},y_{\rho +1},...,y_{\rho +5},y_{j_{b+6}},...,y_{j_{2b+5}}\}$ contains a copy of $\mathcal{H}^l$$\mid$$V(S_1)$. Denote this copy by $\mathcal{S}_1$. Rename $y_{j_1},...,y_{j_b},y_{\rho +1},...,y_{\rho +5},y_{j_{b+6}},...,y_{j_{2b+5}}$ by $a_1,...,a_{2b+5}$ respectively. Clearly $(a_1,...,a_{2b+5})$ is the forest ordering of $\mathcal{H}^l$$\mid$$V(S_1)$, and $T_1$ and $T_2$ are transitive subtournaments of $T$. If $l=1$, we are done. So let us assume that $l\geq 2$.
 
  Fix $k\in \{1,...,l-1\}$ and let $\delta = k(2b+5)$. Assume that there exist vertices  $a_i\in W_{j_i}$ for $i=1,...,\delta$ with $1\leq j_1<...<j_{\delta}\leq k\xi$, such that $T|\{a_1,...,a_{\delta}\}$ contains a copy of $\mathcal{H}^k$ denoted by $\mathcal{H}^{k'}$, $(a_1,...,a_{\delta})$ is the forest ordering of $\mathcal{H}^k$, and every subtournament of $T$ induced by a set of sister leaves of $\mathcal{H}^{k'}$ under $(a_1,...,a_{\delta})$ is a transitive subtournament. Let $t=k+1$ and for $i= k\xi +1,...,k\xi +\xi$, let $W_i^*=\{y\in W_i:\{a_1,...,a_{\delta}\}\rightarrow y\}$. By Lemma \ref{g}, $\mid$$W_i^*$$\mid$ $\geq (1-\frac{\delta}{h})\mid$$W_i$$\mid$ $\geq (1-\frac{1}{2})\mid$$W_i$$\mid$ $\geq \frac{\mid W_i\mid}{2}\geq\frac{c}{2}n $ if $l=2$, $\mid$$W_i^*$$\mid$ $\geq \frac{\mid W_i\mid}{3}\geq\frac{c}{3}n $ if $l=3$, and $\mid$$W_i^*$$\mid$ $\geq (1-\frac{\delta}{h})\mid$$W_i$$\mid$ $\geq (1-\frac{4}{5})\mid$$W_i$$\mid$ $\geq \frac{\mid W_i\mid}{5}\geq\frac{c}{5}n$ otherwise. Assume that $S_t$ is a right-pair-star and assume that $\{v_{\delta+1},...,v_{\delta +b}\}\leftarrow v_{\delta +2b+1}$ (else the argument is similar, and we omit it). Let $Y_1:=\{k\xi +1,...,k\xi +\rho ,k\xi +2^b+1,k\xi +2^b+4\}$. Since $\alpha < \frac{c}{5(\rho +2)}$ and for all $i\in Y_1$, $W^*_{i}\geq \frac{c}{5}n$, then Lemma \ref{q} implies that for all $i\in Y_1$, there exist $y_i\in W^*_i$, such that $\{y_{k\xi +1},...,y_{k\xi +\rho}\}\leftarrow y_{k\xi +2^b+1}\leftarrow y_{k\xi +2^b+4}$. By Theorem \ref{transitive}, $T|\{y_{k\xi +1},...,y_{k\xi +\rho}\}$ contains a transitive subtournament of size $b$.
   Let $T_{2k+1}:=T|\{y_{j_{\delta +1}},...,y_{j_{\delta +b}}\}$ be this transitive subtournament, with $k\xi +1 \leq j_{\delta +1}<...<j_{\delta +b}\leq k\xi +\rho$. Let $Y_2:= \{k\xi +\rho +1,...,k\xi +2^b,k\xi +2^b+2,k\xi +\xi\}$. For all $i\in Y_2$, define $W_i^{**}:=\displaystyle{\bigcap_{x\in A}W^*_{i,x}}$, where $A=V(T_{2k+1})\cup \{y_{k\xi +2^b+1},y_{k\xi +2^b+4}\}$. Then by Lemma \ref{g}, for all $i\in Y_2$, $\mid$$W_i^{**}$$\mid$ $\geq (1-\frac{5(b+2)}{h})\mid$$W_i^*$$\mid$ $\geq \frac{\mid W_i^*\mid}{6}\geq \frac{c}{30}n$. 
   So as $\alpha < \frac{c}{30(\rho +2)}$, then Lemma \ref{q} implies that for all $i\in Y_2$, there exist vertices $y_i\in W_i^{**}$, such that  $\{y_{k\xi +\rho +1},...,y_{k\xi +2^b},y_{k\xi +2^b+2}\}\leftarrow y_{k\xi +\xi}$. By Theorem \ref{transitive}, $T|\{y_{k\xi +\rho +1},...,y_{k\xi +2^b}\}$ contains a transitive subtournament of size $b$. Let $T_{2k+2}:=T|\{y_{j_{\delta +b+1}},...,y_{j_{\delta +2b}}\}$ be this transitive subtournament, with $k\xi +\rho +1 \leq j_{\delta +b+1}<...<j_{\delta +2b}\leq k\xi +2^b$.
    Now for $i=k\xi +2^b+3$, let $W_{i}^{**}=\displaystyle{\bigcap_{x\in B}W^*_{i,x}}$, where $B=V(T_{2k+1})\cup V(T_{2k+2})\cup\{y_{k\xi +2^b+1},y_{k\xi +2^b+2},y_{k\xi +2^b+4},y_{k\xi +2^b+5}\}$. Then by Lemma \ref{g}, $\mid$$W_{i}^{**}$$\mid$ $\geq (1-\frac{2}{3})\mid$$W_{i}^*$$\mid$ $\geq \frac{\mid W_{i}^*\mid}{3}\geq\frac{\mid W_{i}\mid}{6}$ if $l=2$, $\mid$$W_{i}^{**}$$\mid$ $\geq \frac{2\mid W_{i}^*\mid}{3}\geq\frac{2\mid W_{i}\mid}{9}$ if $l=3$, and $\mid$$W_{i}^{**}$$\mid$ $\geq \frac{\mid W_{i}^*\mid}{6}\geq\frac{\mid W_{i}\mid}{30}$ if $l\geq 4$.
     Then $W_{i}^{**}\neq\phi$. Fix $y_{i}\in W_{i}^{**}$ for $i=k\xi +2^b+3$. Rename $y_{j_{\delta +1}},...,y_{j_{\delta +2b}}, y_{k\xi +2^b+1},...,y_{k\xi +2^b+5}$ by $a_{\delta +1},...,a_{t(2b+5)}$ respectively. Now by merging $G_1:=T$$\mid$$\{a_1,....,a_{\delta}\}$ with $G_2:=T$$\mid$$\{a_{\delta +1},...,a_{t(2b+5)}\}$, $G_1\cup G_2$ contains a copy of $\mathcal{H}^t$ and $(a_1,...,a_{t(2b+5)})$ is its forest ordering. Also note that every subtournament in $G_1\cup G_2$ induced by a set of sister leaves of $\mathcal{H}^{t'}$ under  $(a_1,...,a_{t(2b+5)})$, is a transitive subtournament. \\
Now applying this algorithm for $t=2,...,l$ by turn, completes the proof.  $\hfill {\lozenge }$\medbreak
  \noindent Let $\varphi :=(a_1,...,a_{l(2b+5)})$ and let $\gamma_{i} :=(a_{i_1},...,a_{i_{2b+5}})$ be the forest ordering of the copy $\mathcal{S}_i$ of $\mathcal{H}$$\mid$$V(S_i)$ in $T$. Assume first that $S_i$ is a middle-pair-star. We know that $\{a_{i_1},...,a_{i_b}\}$ is adjacent from $p^i_1\in\{a_{i_{b+1}},a_{i_{b+5}}\}$. Let $p^i_2\in \{a_{i_{b+1}},a_{i_{b+5}}\}\backslash \{p^i_1\}$. Note that $p^i_1$ and $p^i_2$ are distinct. For $j=1,2$, let $q^i_j\in \{a_{i_{b+2}},a_{i_{b+4}}\}$, such that $p^i_jq^i_j\in B(\mathcal{S}_i,\gamma )$.  If $\{a_{i_1},...,a_{i_b} \}\rightarrow q^i_1$ and $q^i_2\rightarrow \{a_{i_{b+6}},...,a_{i_{2b+5}} \}$, then do nothing. Otherwise: If $p_1^i=a_{i_{b+1}}$, then there exist $a^1\in\{a_{i_1},...,a_{i_b},a_{i_{b+2}}\}$, $a^2\in\{a_{i_{b+1}},a_{i_{b+5}}\}$, $a^3\in\{a_{i_{b+4}},a_{i_{b+6}},...,a_{i_{2b+5}}\}$, such that $a^1$ is adjacent from $\{a^2,a^3\}$ and $a^2$ is adjacent from $a^3$. And if $p_1^i=a_{i_{b+5}}$, then there exist $a^1\in\{a_{i_1},...,a_{i_{b+1}}\}$, $a^2\in\{a_{i_{b+2}},a_{i_{b+4}}\}$, $a^3\in\{a_{i_{b+5}},...,a_{i_{2b+5}}\}$, such that $a^1$ is adjacent from $\{a^2,a^3\}$ and $a^2$ is adjacent from $a^3$. In this case    delete $\{a_{i_1},...,a_{i_{2b+5}}\}\backslash \{a^1,a^2,a^3\}$ from $\varphi$. Assume now that $S_i$ is a right-pair-star. Let $p^i_1\in \{a_{i_{2b+1}},a_{i_{2b+5}}\}$ such that $\{a_{i_1},...,a_{i_b}\}\leftarrow p_1^i$ and let $p^i_2\in \{a_{i_{2b+1}},a_{i_{2b+5}}\}\backslash \{p^i_1\}$. For $j=1,2$, let $q^i_j\in \{a_{i_{2b+2}},a_{i_{2b+4}}\}$, such that $p^i_jq^i_j\in B(\mathcal{S}_i,\gamma_{i} )$. If $\{a_{i_1},...,a_{i_b} \}\rightarrow q^i_1$ and $\{a_{i_{b+1}},...,a_{i_{2b}} \}\rightarrow q^i_2$, then do nothing. Else there exist three vertices $z^1\in\{a_{i_1},...,a_{i_{2b}}\}$, $z^2\in\{a_{i_{2b+1}},a_{i_{2b+2}}\}$, $z^3\in\{a_{i_{2b+4}},a_{i_{2b+5}}\}$, such that $z^1$ is adjacent from $\{z^2,z^3\}$ and $z^2$ is adjacent from $z^3$. In this case    delete $\{a_{i_1},...,a_{i_{2b+5}}\}\backslash \{z^1,z^2,z^3\}$ from $\varphi$. Finally assume that $S_i$ is a left-pair-star. Let $p^i_1\in \{a_{i_1},a_{i_5}\}$ such that $p_1^i\leftarrow \{a_{i_6},a_{i_{b+5}}\}$ and let $p^i_2\in \{a_{i_1},a_{i_5}\}\backslash \{p^i_1\}$. For $j=1,2$, let $q^i_j\in \{a_{i_2},a_{i_4}\}$, such that $p^i_jq^i_j\in B(\mathcal{S}_i,\gamma_{i} )$. If $q_1^i\rightarrow \{a_{i_6},...,a_{i_{b+5}} \}$ and $q_2^i\rightarrow\{a_{i_{b+6}},...,a_{i_{2b+5}} \}$, then do nothing. Else there exist $r^1\in \{a_{i_1},a_{i_2}\}$, $r^2\in \{a_{i_4},a_{i_5}\}$, $r^3\in \{a_{i_6},...,a_{i_{2b+5}}\}$, such that $r^1$ is adjacent from $\{r^2,r^3\}$ and $r^2$ is adjacent from $r^3$. In this case    delete $\{a_{i_1},...,a_{i_{2b+5}}\}\backslash \{r^1,r^2,r^3\}$ from $\varphi$. Now apply this algorithm for all $i\in \{1,...,l\}$. We get from $\varphi$ a new ordering, say $\tilde{\varphi}:=(a_{r_1},...,a_{r_f})$, with $3l\leq f\leq l(2b+5)$. Let $\hat{\chi}:=(W_{r_1},...,W_{r_f})$. Clearly for $i=1,...,f$, $a_{r_i}\in W_{r_i}$. Let $\varphi^{*}$ be the ordering obtained from $\tilde{\varphi }$ after applying the transitive operation to all sister leaves (if exist) of $T|\{a_{r_1},...,a_{r_f}\}$ under $\tilde{\varphi}$. 
\begin{claim}\label{cl2}
There exists an ordering $\Sigma$ of $H$, satisfying all the following:\\
$\bullet$ $V(H)$ is the disjoint union of $V_1$ and $V_2$ under $\Sigma$.\\
$\bullet$ $H$$\mid$$V_1$is a path-galaxy under $\Sigma_{1}$, the restriction of $\Sigma$ under $V_1$.\\
$\bullet$ $T$$\mid$$\{a_{r_1},...,a_{r_f}\}$ is a copy of $H$$\mid$$V_2$,and $\varphi^{*}$ is the ordering $\Sigma_{2}$, where $\Sigma_{2}$ is the restriction of $\Sigma$ under $V_2$.
\end{claim} 
\noindent \sl {Proof of Claim \ref{cl2}. }\upshape   Define operation $\Upsilon_{1}$ to be the permutation of the vertices $s_{1},...,s_{5}$ that converts the ordering $(s_{1},s_{2},s_{3},s_{4},s_{5})$ to the ordering $(s_{4},s_{1},s_{3},s_{5},s_{2})$. Let $\Theta_{H}(\theta)$ be the set of vertex orderings of $H$ that are obtained from $\theta$ by applying opperation $\Upsilon_{1}$ to the golden vertices of some pair-stars of $H$ under $\theta$. Clearly $\mid$$\Theta_{H}(\theta)$$\mid$ $= 2^l$. Write $\Theta_{H}(\theta):=\{\theta_{1},...,\theta_{2^l}\}$. Fix $i\in \{1,...,2l\}$. For all $\vartheta \in \Theta_{H}(\theta)$, let $\vartheta_{P}$ be the restriction of $V$ to $V_P$, where $H$$\mid$$V_P$ is the path-galaxy under $\vartheta_{P}$ and with maximal order.  Then there exist unique ordering $\Sigma\in \Theta_{H}(\theta)$ such that $V_1=V_P$ and $T$$\mid$$\{a_{r_1},...,a_{r_f}\}$ forms a copy of $H$$\mid$$V_2$, where $V_2=V\backslash V_P$. And moreover $\varphi^{*}$ is the ordering $\Sigma_{2}$, the restriction of $\Sigma$ to $V_2$. This terminates the proof of Claim \ref{cl2}.
  $\hfill {\lozenge }$\medbreak
  \noindent 
 Let $\Sigma_{i}$ be the restriction of $\Sigma$ to $V_i$ for $i=1,2$. Write $\Sigma =(m_1,...,m_{h})$, $\Sigma_{1}= (m_{g_1},...,m_{g_{\kappa}})$, $\Sigma_{2}=(m_{q_1},...,m_{q_f})$. Let $\mathcal{P}:= H$$\mid$$V_1$. We know that $\mathcal{P}$ is a uniform path-galaxy under $\Sigma_{1}$.
    Let $Q_1,...,Q_{\eta}$ be the stars of $\mathcal{P}$ under $\Sigma_{1}$. Let $H^+$ be the tournament obtained from $H$ by replacing $Q_i$ by a star with $2^{b-1}$ leaves for $i=1,...,\eta$. More formal speaking: For all $1\leq i\leq\eta$, let $m_{g_{i_0}}$ be the center of $Q_i$ and $m_{g_{i_1}},...,m_{g_{i_b}}$ be its leaves (clearly $m_{g_{i_1}},...,m_{g_{i_b}}$ are consecutive under $\Sigma$).
     Let $H^+$ be the tournament with $V(H^+)=V(H)\cup (\displaystyle{\bigcup_{i=1}^{\eta}}\{c_1^i,...,c_{\tau}^i\}$ such that 
     $E(B(H^+,\Sigma^{+}))=E(B(H,\Sigma ))\cup (\displaystyle{\bigcup_{i=1}^{\eta}}\{m_{g_{i_0}}c_1^{i},...,m_{g_{i_0}}
     c^i_{\tau}\})$, where $\Sigma^{+}$ is the ordering obtained from $\Sigma$ after inserting $c_1^i,...,c_{\tau}^i$ just 
     before $m_{g_{i_1}}$ in $\Sigma$ for all $1\leq i\leq\eta$, where $\tau = 2^{b-1}-b$.  
     Rename $\Sigma^{+}$ as follows: $\Sigma^{+}=(n_1,...,n_{h+\eta \tau})$. Let $\Sigma^{+}_{2}:=(n_{h_1},...,n_{h_f})$ be the restriction of $\Sigma^{+}$ to $V_2$. Rename $W_{r_i}$ by $N_{h_i}$ for $i=1,...,f$. 
   Now enrich $\hat{\chi}=(N_{h_1},...,N_{h_f})$ by $p:=\kappa +\eta \tau$ sets from $\chi$ after renaming them by $N_{e_1},...,N_{e_{p}}$, in a way that the outcome will be $(N_1,...,N_{h+\eta\tau})$ and $\hat{\chi}^+:= (N_1,...,N_{h+\eta\tau})$ is a smooth $(c,\frac{1}{h},\overline{w})$-structure, where $\overline{w}$ is an all-zero vector of length $h+\eta\tau$. We can do that because for any $i\in\{1,...,f-1\}$, there exist in $\chi$ at least $\nu$ sets lying between $N_{h_i}$ and $N_{h_{i+1}}$. Clearly $\Sigma^{+}_{1}=(n_{e_1},...,n_{e_p})$ is the restriction of $\Sigma^{+}$ to $V_1^+:=V(H^{+})\backslash V_2$, and $H_1^+:= H^+|V_1^+$ is a path-galaxy under $\Sigma_{1}^+$. Let $Q_1^+,...,Q_{\eta}^+$ be the stars of $H_1^+$ under $\Sigma_{1}^{+}$.  For simplicity, for $i=1,...,p$, let us rename $N_{e_i}$ and $n_{e_i}$ by $E_i$ and $w_i$ respectively. For all $1\leq i\leq \eta$, let $w_{i_o}$ be the center of $Q_i^+$ and $w_{i_1},...,w_{i_\rho}$ be its leaves.
    
 \begin{claim}\label{cl3}
 For all $1\leq i\leq \eta$, there exist $b$ sets $E_{i_{\iota_{1}}},...,E_{i_{\iota_{b}}}$, with $i_1\leq i_{\iota_{1}}<...<i_{\iota_{b}}\leq i_{\rho}$, such that $\mathcal{P}$ is well-contained in $\chi^{*}$, the smooth $(c,\lambda ,w^*)$-structure containing the sets in $E:=\displaystyle{\bigcup_{i=1}^{i=\eta}}\{E_{i_0},E_{i_{\iota_{1}}},...,E_{i_{\iota_{b}}}\}$ (note that the sets of $E$ are ordered in $\chi^{*}$ according to their appearance in $\chi$), where $w^*$ is an all-zero vector of length $\eta (b+1)$.
 \end{claim}
 \noindent \sl {Proof of Claim \ref{cl3}. }\upshape 
   Let $Y:= \{a_{r_1},...,a_{r_f}\}$ and for $i=1_0,...,1_{\rho}$, let $E_i^*=\displaystyle{\bigcap_{x\in Y}E_{i,x}}$. By Lemma \ref{g}, for all $i\in \{1_0,...,1_{\rho}\}$, $\mid$$E_i^{*}$$\mid$ $\geq (1-\frac{f}{h})\mid$$E_i$$\mid$ $\geq\frac{2\mid E_i\mid}{9}\geq\frac{ 2c}{9}n$. Denote $\frac{ 2c}{9}$ by $\hat{c}$. For simplicity, for $i=1,...,\kappa$, let us rename $m_{g_i}$ by $d_i$. For $k \in \lbrace 1,...,\eta \rbrace$, define $\mathcal{P}^{k} = \mathcal{P}$$\mid$$\bigcup_{j=1}^{k} V(Q_{j})$, where $\mathcal{P}^{\eta} = \mathcal{P}$.  The proof works as follows: we will construct $\mathcal{P}$ star by star after updating the sets corresponding to each star in order to merge it with the previously constructed stars. Let $\Sigma_{1}^{k}:=(d_{k_1},...,d_{k_q})$ be the restriction of $\Sigma_{1}$ to $V(\mathcal{P}^k)$, where $k\in\{1,...,\eta\}$ and $q=k(b+1)=\mid$$\mathcal{P}^k$$\mid$.
   As $\alpha < \frac{\hat{c}}{\rho +1}$ and $E^*_{1_i}\geq \hat{c}n$ for $i=0,...,\rho$, then Lemma \ref{q} implies that there exist $x_{1_i}\in E_{1_i}^*$ for $r=0,...,\rho$, such that $x_{1_0}\leftarrow \{x_{1_1},...,x_{1_{\rho}}\}$ if $Q^+_1$ is a left star, and $\{x_{1_1},...,x_{1_{\rho}}\}\leftarrow x_{1_0}$ if $Q^+_1$ is a right star. By Theorem \ref{transitive}, $T|\{x_{1_1},...,x_{1_{\rho}}\}$ contains a transitive subtournament of size $b$. So there exists $b$ vertices $x_{1_{\iota_{i}}}\in E_{1_{\iota_{i}}}$ for $i=1,...,b$, with $1_1\leq 1_{\iota_{1}}<...<1_{\iota_{b}}\leq 1_{\rho}$, and such that $T|\{x_{1_{\iota_{1}}},...,x_{1_{\iota_{b}}}\}$ is a transitive subtournament. Let $\Lambda_{1}$ be the ordering of $\{x_{1_0},x_{1_{\iota_{1}}},...,x_{1_{\iota_{b}}}\}$ according to their appearance in $\chi$. Then by applying the transitive operation to $\{x_{1_{\iota_{1}}},...,x_{1_{\iota_{b}}}\}$ transforms $\Lambda_{1}$ to the star ordering of $Q_1$ (i.e the restriction $\Sigma_{1}^{1}$ of $\Sigma_{1}$ to $V(Q_1)$). Hence $\Lambda_{1}\in \Pi_{\mathcal{P}^{1}}(\Sigma_{1}^{1})$.  Fix $k\in \{1,...,\eta -1\}$ and assume that for $i=1,...,k$, there exist $b$ sets $E_{i_{\iota_{1}}},...,E_{i_{\iota_{b}}}$, with $i_1\leq i_{\iota_{1}}<...<i_{\iota_{b}}\leq i_{\rho}$, such that $\mathcal{P}^k$ is well-contained in $\chi_{k}^{*}$, the smooth $(c,\lambda ,w_k^*)$-structure containing the sets in $E:=\displaystyle{\bigcup_{i=1}^{i=k}}\{E_{i_0},E_{i_{\iota_{1}}},...,E_{i_{\iota_{b}}}\}$. Denote by $P_k$ the well-contained copy of $\mathcal{P}^k$.  Let $t=k+1$. For $i=0,...,\rho$, let $E_{t_i}^{*}=\displaystyle{\bigcap_{x\in Y\cup P_k}E_{t_i,x}}$. By Lemma \ref{g}, for all $i\in \{0,...,\rho\}$, $\mid$$E_{t_i}^{*}$$\mid$ $\geq (1-\frac{f+k(b+1)}{h})\mid$$E_{t_i}$$\mid$ $\geq\frac{(b+1)\mid E_{t_i}\mid}{h}\geq\frac{(b+1)c}{h}n$. As $\alpha < \frac{(b+1)c}{h(\rho +1)}$ and $E^*_{t_i}\geq \frac{b+1}{h}cn$ for $i=0,...,\rho$, then Lemma \ref{q} implies that there exist $x_{t_i}\in E_{t_i}^*$ for $i=0,...,\rho$, such that $x_{t_0}\leftarrow \{x_{t_1},...,x_{t_{\rho}}\}$ if $Q^+_t$ is a left star, and $\{x_{t_1},...,x_{t_{\rho}}\}\leftarrow x_{t_{0}}$ if $Q^+_t$ is a right star. By Theorem \ref{transitive}, $T|\{x_{t_1},...,x_{t_{\rho}}\}$ contains a transitive subtournament of size $b$. So there exists $b$ vertices $x_{t_{\iota_{i}}}\in E_{t_{\iota_{i}}}$ for $i=1,...,b$, with $t_1\leq t_{\iota_{1}}<...<t_{\iota_{b}}\leq t_{\rho}$, and such that $T|\{x_{t_{\iota_{1}}},...,x_{t_{\iota_{b}}}\}$ is a transitive subtournament. Now merging $P_k$ with  $T|\{x_{t_0},x_{t_{\iota_{1}}},...,x_{t_{\iota_{b}}}\}$ we get a well-contained copy of $\mathcal{P}^k$ and moreover the ordering of the vertices of $P_t$ according to their appearance in $\chi$ belongs to $\Pi_{\mathcal{P}^{t}}(\Sigma_{1}^{t})$, where $P_t$ is the well-contained copy of $\mathcal{P}^t$, and $\Sigma_{1}^{t}$ is the restriction of $\Sigma_{1}$ to $V(\mathcal{P}^t)$. Applying this algorithm for $t=2,...,\eta$ completes the proof. $\hfill {\lozenge }$\medbreak
  \noindent Now by merging the well-contained copy of $\mathcal{P}$ with $T$$\mid$$Y$ we get a tournament in $T$ which is isomorphic to $H$. So $T$  contains $H$, a contradiction. This completes the proof. $\hfill {\square }$
\end{proof}

\end{document}